\documentclass{article}
\usepackage{amssymb}
\usepackage{amsmath}
\usepackage{amsfonts}
\usepackage{amsthm}
\usepackage{color}

\newtheorem{theorem}{Theorem}
\newtheorem{proposition}{Proposition}

\def\nil{{\mathrm{Nil}\, }}
\def\sl{\widetilde{\mathrm{PSL}_2 (\mathbb{R})}}

\begin{document}

\thispagestyle{empty}

\title{Willmore--like functionals for surfaces in  $3$--dimensional Thurston geometries
\let\thefootnote\relax
\author{D.~Berdinsky
 \and Y.~Vyatkin
 }
}


\date{}
\maketitle {\small
\begin{quote}
\noindent{\sc Abstract.  } We find analogues of the Willmore functional for each of the Thurston geometries with $4$--dimensional isometry group such that the CMC--spheres in these geometries are critical points of these functionals.

\end{quote}
}

 \section{Introduction}

  Let $M$ be a closed orientable surface and $f \colon M \to N$ be an immersion of $M$
  into a $3$--dimensional Riemannian manifold $N$.
  Set:
  \begin{equation*}
     {\cal W} (f) =    \int_{M} \left( H^2 + \overline{K} \right) d \mu,
  \end{equation*}
  where $H$ is the mean curvature of the immersed surface,  the value of $\overline{K}$
  at a point $p \in M$ is defined as the sectional curvature of the $2$--plane $f_* (T_p M)$ in $N$, $d \mu$ is the area element
  of the induced metric on $M$.
  We will refer to the functional ${\cal W}$ as the Willmore functional.
  It is known that ${\cal W} (f)$ is a conformal invariant \cite{Weiner}.

  In the $3$--dimensional space forms $\mathbb{R}^3, \mathbb{H}^3$ and $\mathbb{S}^3$ the functional ${\cal W}$
  enjoys the property that the CMC spheres are critical points of
  ${\cal W}$; recall that in the $3$--dimensional space forms the CMC spheres are exactly the round spheres by the Hopf
  theorem. However, this property for the Willmore functional ${\cal W}$ fails to hold in the other $3$--dimensional Thurston
  geometries.

  In this paper we will introduce the Willmore--like functionals
  for the certain family of Riemannian manifolds $E(k,\tau)$ that include
  the model spaces for all Thurston geometries with $4$--dimensional group of isometries,
  i.e., the products $\mathbb{S}^2 \times \mathbb{R}$ and $\mathbb{H}^2 \times
  \mathbb{R}$, the Heisenberg group $\nil$ and the Lie group $\sl$.
  The functionals to be introduced in these geometries have the
  form:
  \begin{equation}
  \label{willmintrogen}
        \int_{M} \left( H^2 + \alpha \overline{K} + \beta \right) d
        \mu,
  \end{equation}
  where $\alpha$ and $\beta$ are some constants that depend on $k$ and
  $\tau$.
  In the case of the Heisenberg group $\nil$ (for $k=0$ and $\tau = \frac{1}{2}$) the
  functional:
  \[
       \int_{M} \left(  H^2 +  \frac{1}{4} \overline{K} -
       \frac{1}{16} \right) d \mu
  \]
  was obtained in \cite{BT05} based on the Weierstrass representation for surfaces in $\nil$.
  Then it was shown \cite{BT07} that the CMC
  spheres in $\nil$ are critical points of this functional.
  In the case of the Lie group $\sl$ (for $k = -1$ and $\tau = -
  \frac{1}{2}$) it was shown that for the functional:
   \[
       \int_{M} \left(  H^2 +  \frac{1}{4} \overline{K} -
       \frac{5}{16} \right) d \mu
  \]
  the minimum among the rotationally invariant spheres is attained
  exactly at the CMC spheres~\cite{Berd10}.

 The main result of the paper is the following theorem.
 \begin{theorem} \label{thm:main}
  The CMC spheres in $E(k,\tau)$ are critical points of the following
 Willmore--like  functional:
 \begin{equation}
 \label{Eosnovformula}  
   E(f) = \int_M \left( H^2 + \frac{1}{4} \overline{K} + \frac{k}{4} - \frac
   {\tau^2}{4} \right) d\mu.
 \end{equation}
 \end{theorem}
 In addition to Theorem \ref{thm:main} we will prove the following theorem.  
 \begin{theorem} \label{thm:main2}
    For rotationally invariant spheres in $E(k,\tau)$ 
    the functional $E(f)$ attains its minimum exactly at the CMC spheres. 
 \end{theorem} 
\smallskip 
\smallskip
\noindent {\sc Remark 1.}
 We note that:
\begin{equation*}
   E(f) = {\cal W}(f) + \int_M \left( -\frac{3}{4} \overline{K} + \frac{k}{4} - \frac{\tau^2}{4} \right)
   d\mu.
\end{equation*}

\noindent {\sc Remark 2.}
  It can be seen that Theorem \ref{thm:main} agrees with  the
  results obtained earlier for $\nil$ \cite{BT07} and the Lie group $\sl$
  \cite{Berd10}. In addition, the functional $E$ for the case $\mathbb{S}^2 \times \mathbb{R}$ ($k=1$ and $\tau = 0$)
   coincides up to a constant factor with the
  functional \footnote{In \S~6.2, \cite{BT07} the term $H$ should be considered as $2H$.}
  $\int_M (4H^2 +  \overline{K} + 1)$  mentioned in \S~6.2, \cite{BT07}.

\smallskip
\smallskip

The structure of the remaining part of this paper is as follows.
In \S~\ref{sectionmodelspaces} we give the description of the
Riemannian manifolds $E(k,\tau)$. In \S~\ref{sectioncmcspheres} we review the
characterizations of the CMC spheres in these manifolds.
In \S~\ref{sectionproofoftheorem}  we give the details of the proof
of Theorem \ref{thm:main}. 
In \S~\ref{proofoftheorem2} we give the details of the proof of 
Theorem \ref{thm:main2}.

 \section{The Riemannian manifods $E(k, \tau)$}
 \label{sectionmodelspaces}

 The model spaces for the four Thurston geometries:
 $\nil$, $\sl$, $\mathbb{H}^2 \times \mathbb{R}$ and $\mathbb{S}^2 \times
 \mathbb{R}$ belong to the family of Riemannian
 3--manifolds $E(k,\tau), k \in \mathbb{R}, \tau \in \mathbb{R}$ that are  as follows.
 If $k \geqslant 0$ then $E(k,\tau)$ is $\mathbb{R}^3$ with the
 metric:
 \begin{equation}
 \label{metricekt}
    ds^2 = \frac{dx^2 + dy ^2}{\left( 1 + \frac{k}{4} (x^2 + y^2)
    \right)^2} + \left(dz + \frac{\tau (y dx - x dy )}{1 + \frac{k}{4} (x^2 + y^2
    )}\right)^2.
 \end{equation}
 If $k < 0$ then $E(k,\tau)$ is the product $\mathrm{D}^2(\frac{2}{\sqrt{-k}}) \times \mathbb{R}$ with the metric \eqref{metricekt},
 where  $\mathrm{D}^2 (\frac{2}{\sqrt{-k}}) =\{ (x,y) \, | \,  x^2 + y^2 <  \frac{4}{-k}$ \}.
 The family $E(k,\tau)$ is also referred to as
 Bianchi--Cartan--Vranceanu family \cite{Inoguchi02,Inoguchi13}.
 The projection of $E(k,\tau)$ onto the $2$--dimensional domain of constant curvature $k$ given by the map $(x,y,z)\mapsto (x,y)$
 is a Riemannian fibration.
 The fibres of such a fiber bundle are geodesics and its unitary tangent vectors $\frac{\partial}{\partial z}$
 form a Killing vector filed; this field is also referred to as the vertical vector field.
 The parameter $k$ is called the base curvature and $\tau$ the
 bundle  curvature.

 If $k = -1 , \tau = 0 $ then $E (k, \tau)$ is  the product $\mathbb{H}^2 \times
 \mathbb{R}$.
 If $k = 1, \tau = 0$ then $E (k, \tau)$ is obtained from the product $\mathbb{S}^2 \times
 \mathbb{R}$ by removing one fibre.
 If $k = 0, \tau \neq 0$ then $E (k,\tau)$ is the Heisenberg group $\nil$ with the left--invariant
 metric determined by the parameter $\tau$.
 If $k < 0, \tau \neq 0$ then $E (k, \tau)$ is the Lie
 group $\sl$ with the left--invariant metric determined by the
 parameters $k$ and $\tau$.
 For the case $k > 0 , \tau \neq 0$, the manifolds $E (k, \tau)$
 are obtained from the covering of the Berger spheres by removing one fibre.

 For more details we refer the reader to \cite{Scott83,Da07}.
 We will need the following proposition.
 \begin{proposition}
 \label{propseccurv}
 The sectional curvature of a $2$--plane in $E(k,\tau)$ equals:
 \begin{equation}
 \label{seccurv}
     \overline{K}=\tau^{2}+(k-4\tau^{2}) \nu^{2},
 \end{equation}
 where $\nu$ is the scalar product of a unit normal vector to the plane  and the vertical vector
 $\xi = \frac{\partial}{\partial z}$ with respect to the metric
 \eqref{metricekt}.
 \end{proposition}
 \begin{proof}
   The identity \eqref{seccurv} can be obtained directly from the general formula
   for the Riemann curvature tensor of $E(k,\tau)$ shown in Proposition~2.1, \cite{Da07}.
 \end{proof}

 \section{The CMC spheres in $E(k,\tau)$}
 \label{sectioncmcspheres}

 The rotationally invariant CMC surfaces in the products
 $\mathbb{H}^2 \times \mathbb{R},\mathbb{S}^2 \times
 \mathbb{R}$ and the Heisenberg group $\nil$  were described in \cite{HH,PR} and \cite{Tomter93,FMP99,CPR} respectively.
 The case of the Lie group $\sl$ and the Berger spheres were studied
 in \cite{Esp09} and \cite{Torralbo10}.
 In order to describe rotationally invariant CMC surfaces in $E(k,\tau)$ we will follow the approach used in
 \cite{Tomter93,FMP99} for $E(0,\frac{1}{2})$.


 For the cylindrical coordinates $\rho, \theta, z$ in $E(k,\tau)$ such that  $ x= \rho \cos \theta, y = \rho \sin
 \theta, z=z$ the metric \eqref{metricekt}  has the form:
 \begin{equation}
 \label{metric2-sl2}
 ds^2 =
 \frac{1}{(1+\frac{k}{4}\rho^{2})^{2}}d\rho^{2}+\frac{\rho^{2}+\tau^{2}\rho^{4}}{(1+\frac{k}{4}\rho^{2})^{2}}d\theta^{2}-
 \frac{2\tau \rho^{2}}{1+\frac{k}{4}\rho^{2}}dzd\theta+dz^{2}.
 \end{equation}
 We note that $\rho \in [0, R)$, where $R = \frac{2}{\sqrt{-k}}$ if $k < 0$
 and $R = \infty$ if $k \geqslant 0$.

 The group $\mathrm{SO}(2)$ acts on  $E(k,\tau)$  by rotations
 $ \theta \mapsto \theta +  const$ around $z$--axis.
 The rotations are isometries and the factor--space
 $E (k,\tau) \slash \mathrm{SO}(2)$ is the $2$--dimensional domain
 $\mathrm{B} (k,\tau) = \lbrace (u, v)| u \in [0, R ), v \in \mathbb{R} \rbrace$ with the metric:
 \begin{equation}
 \label{metric-half-plane}
  d \widetilde{s}^2 = \frac{1}{(1+\frac{k}{4}u^{2})^{2}}du^{2}+\frac{1}{1+\tau^{2}u^{2}}dv^{2},
 \end{equation}
 so the projection
 $E(k,\tau) \rightarrow \mathrm{B}(k,\tau)$ is a Riemannian submersion.

 For a given rotationally invariant surface we define by $\gamma(s)=(u(s),v(s))$ its projection onto $\mathrm{B} (k,\tau)$,
 where $s$ is a natural parameter with respect to the metric \eqref{metric-half-plane}.
 Let $\sigma$ be the angle between $\dot{\gamma}$ and $\frac{\partial}{\partial
 u}$. It can be verified (cf. \cite{FMP99}, eq.~2) that for the metric \eqref{metric-half-plane}
 the geodesic curvature of $\gamma(s)$  equals:
 \begin{equation}
 \label{geodcurv2}
    \widetilde{k} =
    \dot{\sigma}-\frac{\tau^{2}u(1+\frac{k}{4}u^{2})}{(1+\tau^{2}u^{2})}\sin\sigma,
 \end{equation}
 The mean curvature of a rotationally invariant surface is given by
 the reduction theorem (cf. \cite{FMP99}, p.~178) as follows:
 \begin{equation}
 \label{meancurv}
    H = \frac{1}{2} \left( \widetilde{k} - \frac{\partial}{\partial n} \ln \mu \right),
 \end{equation}
 where $n = (-(1+\frac{k}{4}u^2) \sin \sigma, \sqrt{1+\tau ^2 u^2} \cos \sigma)$ is a normal vector in $\mathrm{B}(k, \tau)$ to $\gamma(s)$  and
 $\mu  = \frac{u \sqrt{1+\tau^2 u^2}}{1+\frac{k}{4}u^2}$ is the
 factor of the volume form for an $\mathrm{SO}(2)$ orbit with respect to the metric \eqref{metric2-sl2}.
 From \eqref{geodcurv2} and \eqref{meancurv} we obtain:
 \begin{equation}
 \label{meancurv2}
   H= \frac{1}{2} \left( \dot{\sigma} + \left(\frac{1}{u} - k \frac{u}{4}
   \right) \sin \sigma \right).
 \end{equation}
 Thus we obtain that
 for a profile $\gamma (s) = (u(s), v(s))$ of a rotationally invariant CMC surface
 the following system of ODE is satisfied:
 \begin{equation}
 \label{ode_sl2}
 \begin{cases}
 \dot{u} = \left( 1+\frac{k}{4} u^2 \right) \cos \sigma, \\
 \dot{v} = \sqrt{1+ \tau^2 u^2 } \sin \sigma, \\
 \dot{\sigma} = 2H - \left( \frac{1}{u}- k \frac{u}{4} \right)  \sin
 \sigma.
 \end{cases}
 \end{equation}
 It can be straightforwardly verified that the system \eqref{ode_sl2} has the following first integral:
 \begin{equation}
 \label{firstintegral}
 J = \frac{u}{1+\frac{k}{4} u^2} \left( \sin \sigma - Hu \right).
 \end{equation}
 Then we have the following proposition.
 \begin{proposition}
 \label{rotinvcmcsphery}   If $k \leqslant 0$ then for any $H$ such that $H^2 > \frac{-k}{4}$ there exists
    a rotationally invariant CMC sphere of constant mean curvature $H$ in $E(k,\tau)$;
    moreover, if $H^2 \leqslant \frac{-k}{4}$ then there exists no CMC  sphere of constant mean curvature $H$ in $E(k,\tau)$.
    If $k > 0$ then for any $H \neq 0$ there exists a rotationally invariant CMC sphere of constant mean curvature $H$ in $E(k,\tau)$.
    For every rotationally invariant CMC sphere in  $E(k,\tau)$ the first integral \eqref{firstintegral} vanishes: $J=0$.
    The CMC spheres in $E(k,\tau)$ are unique up to isometries.
 \end{proposition}
 \begin{proof} The proof is based on an analysis of a qualitative behavior
  of the solutions of \eqref{ode_sl2} depending on the values of $J$ and $H$.
  Such an analysis is straightforward and it was done for $E(0,\frac{1}{2})$ in \cite{Tomter93,FMP99};
  the case of $E(-1,-\frac{1}{2})$  was shown in \cite{Berd10}.
  The uniqueness of the CMC spheres was proved in \cite{AR04,AR05}.
  Also, see \cite{FM10} for the complete proof of Proposition \ref{rotinvcmcsphery}.  
  \end{proof}
 By  Proposition \ref{rotinvcmcsphery}  we obtain that on a rotationally invariant CMC sphere in $E(k,\tau)$
 the following equality  holds:
 \begin{equation}
 \label{sinussigmaravnohu}
    \sin \sigma  = H u.
 \end{equation}
 
 \smallskip
 \smallskip
 
 \noindent {\sc Remark 3.}
    Although there exists no a minimal sphere in $E(k,\tau)$ such
    spheres exist for $\mathbb{S}^2 \times \mathbb{R}$ and the
    Berger spheres. We recall that for $k>0$ the manifolds $E(k,\tau)$
    are obtained from the corresponding homogeneous manifolds by
    removing one fibre.
 
 \section{The proof of Theorem \ref{thm:main}}
 \label{sectionproofoftheorem}

 For an immersion $f: M \rightarrow E(k,\tau)$ of a closed orientable surface $M$ into $E(k,\tau)$ set:
 \begin{equation}
  \label{genwillm1}
     E_{\alpha, \beta} (f) = \int_M \left( H^2 + \alpha \overline{K} + \beta \right) d \mu.
 \end{equation}

 Let $F : M \times [0,1] \rightarrow E(k,\tau)$ be a normal variation of
 the immersion $f$, i.e., $F(p,0) = f(p)$ for all $p \in M$ and $\frac{\partial F (p,t)}{\partial t} = \varphi n$, where $n$ is
 the unit normal vector field to $M$ and the velocity $\varphi$ is a smooth function on $M$.
 We will denote by $\delta$ the operator $\frac{\partial}{\partial t}|_{t=0}$.
 We will need the following proposition.
 \begin{proposition}
 Under a normal variation with the velocity $\varphi$ the following identities hold:
 \begin{equation} \label{deltamu} \delta d \mu  = - 2 H \varphi d \mu, \end{equation}
 \begin{equation} \label{deltan} \delta n = - \nabla \varphi, \end{equation}
 \begin{equation} \label{deltaH} 2 \delta H = \Delta \varphi + (4H^2 - 2K_e + \mathrm{Ric}(n,n)) \varphi, \end{equation}
 where $\nabla$ is the gradient  and $\Delta$ is the Laplace--Beltrami operator on $M$,
 $K_e$ is the extrinsic Gauss curvature.
 \end{proposition}
 \begin{proof}
 The proof is standard, one may look it up in \cite{Huisken99}.
\end{proof}
 It follows from Proposition \ref{propseccurv} that
 the term $\mathrm{Ric}(n,n)$ equals
 $k -2 \tau^2 - (k - 4 \tau^2)  \nu ^2$.
 Therefore we have that:
 \begin{equation}
 \label{variationH1}
   2 \delta H =  \Delta \varphi  + (4H^2 - 2 K_e   +  k -2 \tau^2 - (k - 4 \tau^2)  \nu ^2) \varphi.
 \end{equation}
 Let $T$ be the projection of the vertical field $\xi$ on $M$, i.e., $T = \xi - \nu n$.
 By \eqref{deltan} we have that
 $\delta \nu = \delta \langle n, \xi \rangle = - \langle \nabla \varphi , \xi
 \rangle$. Therefore we obtain:
 \begin{align}
 \label{deltanu}
   \int_M \nu \delta \nu  d \mu = \int_M  \mathrm{div} (\nu T) \varphi d \mu.
 \end{align}
 By Proposition \ref{propseccurv} we have:
 \begin{equation}
    E_{\alpha, \beta} (f) = \int_M (H^2 + \alpha (k - 4 \tau ^2) \nu ^2 + \beta + \alpha \tau^2) d \mu.
 \end{equation}
 Then by \eqref{variationH1},\eqref{deltanu} and \eqref{deltamu} we obtain:
 \begin{align}
 \delta E_{\alpha,\beta} (f) =   \int_M ( & \Delta H +  H ( 2 H^2  - 2 K_e - (1+ 2 \alpha) (k -4\tau^2)
    \nu     ^2   +  \\
    & + k -2 \tau^2 - 2 \beta  -2 \alpha \tau^2 ) + 2 \alpha (k - 4 \tau ^2) \mathrm{div} (\nu T)) \varphi d
    \mu. \nonumber
\end{align}
Therefore the Euler--Lagrange equation of the functional
\eqref{genwillm1} is as follows:
 \begin{align}
 \label{euler1}
     \Delta H +  H (  & 2 H^2  - 2 K_e - (1+ 2 \alpha) (k -4\tau^2)
    \nu     ^2   + \\
    & + k -2 \tau^2 - 2 \beta  -2 \alpha \tau^2 ) + 2 \alpha (k - 4 \tau ^2) \mathrm{div} (\nu
    T) = 0. \nonumber
\end{align}
 By the Gauss theorem we obtain $ K_e = K - \overline{K} = K -
 (k- 4 \tau^2) \nu ^2 -
 \tau^2$, where $K$ is the intrinsic Gauss curvature.
 Then  \eqref{euler1} can be rewritten as follows:
 \begin{equation}
 \label{euler2}
 \begin{split}
    \Delta H + H ( & 2 H^2  - 2 K + (1- 2 \alpha) (k -4\tau^2) \nu ^2
    + \\
     & + k  - 2 \beta - \alpha \tau^2 ) +  2 \alpha (k - 4 \tau ^2) \mathrm{div} (\nu
      T)  = 0.
 \end{split}
 \end{equation}

 Consider a CMC sphere in $E(k,\tau)$. For the coordinates on this sphere we choose
 $\theta$ and $s$; recall that $\theta$ is an angle from the cylindrical coordinate system in $E(k,\tau)$ and
 $s$ is the natural parameter on the projection $\gamma(s) = (u(s),v(s))$ of this sphere onto $\mathrm{B}(k,\tau)$.
 By \eqref{metric2-sl2} for these coordinates
 the metric on a CMC sphere is as
 follows:
 \begin{equation}
 \label{metricsphere3}
 \frac{\rho^{2}+\tau^{2}\rho^{4}}{(1+\frac{k}{4}\rho^{2})^{2}} d\theta ^2 + ds^2.
 \end{equation}
 By \eqref{metric-half-plane}, \eqref{ode_sl2} and \eqref{metricsphere3} it can be straightforwardly verified that on a CMC sphere in $E(k,\tau)$:
 \begin{equation}
 \label{nucmc}
   \nu = \frac{\cos \sigma} {\sqrt{1 + \tau^2 u^2}},
 \end{equation}
 \begin{equation}
 \label{kcmc}
    K =
   -\frac{1+\frac{k}{4} u^2}{u\sqrt{1+\tau^2 u^{2}}}\frac{d}{ds^{2}}\frac{u\sqrt{1+\tau^2 u^2}}{1+ \frac{k}{4} u^2},
 \end{equation}
 \begin{equation}
 \label{divcmc}
    \mathrm{div} (\nu
      T) =
     \frac{1+\frac{k}{4}u^2}{u\sqrt{1+\tau^2 u^2}}
     \frac{d}{ds} \frac{u \cos \sigma \sin \sigma}{(1+\frac{k}{4} u^2)\sqrt{1+ \tau^2
     u^2}}.
 \end{equation}

  Put $\alpha = \frac{1}{4}$ and $\beta = \frac{k}{4}-
 \frac{\tau^2}{4}$. Substituting \eqref{nucmc},\eqref{kcmc} and \eqref{divcmc} into the left hand side of \eqref{euler2} we obtain that it equals:
 \begin{equation}
 \label{leftsideeuler1}
 \begin{split}
    2 H^3 + H \left( 2 \frac{1+\frac{k}{4} u^2}{u\sqrt{1+\tau^2 u^{2}}}\frac{d}{ds^{2}}\frac{u\sqrt{1+\tau^2 u^2}}{1+ \frac{k}{4}
    u^2} + \frac{1}{2} (k -4\tau^2)  \frac{\cos ^2  \sigma} {1 + \tau^2
    u^2} +
     \frac{k}{2} \right) + \\
     +\frac{1}{2}(k - 4 \tau ^2) \frac{1+\frac{k}{4}u^2}{u\sqrt{1+\tau^2 u^2}}
     \frac{d}{ds} \frac{u \cos \sigma \sin \sigma}{(1+\frac{k}{4} u^2)\sqrt{1+ \tau^2
     u^2}}
 \end{split}
 \end{equation}
 Using the equation \eqref{sinussigmaravnohu} and the system
 \eqref{ode_sl2} it can be verified that the expression \eqref{leftsideeuler1}
 vanishes on a CMC sphere. Theorem \ref{thm:main} is proved.

 \section{The proof of Theorem \ref{thm:main2}}
 \label{proofoftheorem2}
 
 Let us substitute \eqref{seccurv} into the formula \eqref{Eosnovformula} for
 the functional $E(f)$. Then we have:  
  \begin{equation}
  \label{Eosnovformulachereznu}  
  E(f) = \int_M  
          \left( H^2 +\left(\frac{k}{4}-\tau^{2}\right) \nu^{2} + \frac{k}{4}            
          \right)d\mu.          
  \end{equation} 
  Let us consider a rotationally invariant sphere in $E(k,\tau)$ 
  defined by a curve $\gamma (s) = (u(s),v(s)) \subset \mathrm{B} (k,\tau)$. 
  By \eqref{meancurv2} we can represent $H^2$ as follows: 
  \begin{equation} 
  \label{hsquarerotinv}   
     H^2 =  \frac{1}{4} \left( \dot{\sigma} - 
     \left( \frac{1}{u} + k \frac{u}{4} \right)  \sin \sigma \right)^2  + 
      \left( \frac{\dot{\sigma}  \sin \sigma }{u} - 
       k \frac{\sin^2 \sigma }{4}\right).
  \end{equation}
  For a rotationally invariant surface we have: 
  \begin{equation}  
  \label{nusquarerotinv}   
     \nu^2  =   \frac{\cos^2 \sigma}{1 + \tau^2 u ^2},
  \end{equation}
  \begin{equation}
  \label{dmurotinv}   
     d \mu =  \frac{u \sqrt{1+\tau^2 u^2}}{1+\frac{k}{4}u^2} ds.
  \end{equation}
  Substituting \eqref{hsquarerotinv}, \eqref{nusquarerotinv} and 
  \eqref{dmurotinv} into \eqref{Eosnovformulachereznu} we obtain: 
 \begin{equation}  
 \label{Erotinv}
 \begin{split}   
    E(f) = 2 \pi \int_\gamma  \frac{1}{4} \left( \dot{\sigma} - 
    \left( \frac{1}{u} + k \frac{u}{4} \right)  \sin \sigma \right)^2  \frac{u \sqrt{1+\tau^2 u^2}}{1+\frac{k}{4}u^2} ds +  \\
    2 \pi \int_\gamma  \left( \frac{\dot{\sigma}  \sin \sigma }{u} - 
    k \frac{\sin^2 \sigma }{4}
    + \left(\frac{k}{4}-\tau^{2}\right) \frac{\cos^2 \sigma}{1 + \tau^2 u ^2}  + \frac{k}{4}
    \right) \frac{u \sqrt{1+\tau^2 u^2}}{1+\frac{k}{4}u^2} ds.
 \end{split} 
 \end{equation}  
  By \eqref{ode_sl2} we have: $\cos \sigma = 
  \frac{\dot{u}}{1 + \frac{k}{4} u^2}$. Substituting this into the integrand of the second summand in \eqref{Erotinv} we obtain: 
  \begin{equation}
  \label{secondsummandintegrand}   
     -\frac{\ddot{u}\sqrt{1+\tau^{2}u^{2}}}{(1+\frac{k}{4}u^{2})^{2}}+\frac{u\dot{u}^{2}}{(1+\frac{k}{4}u^{2})^{3}}\left(\frac{3}{4}k\sqrt{1+\tau^{2}u^{2}}+\left(\frac{k}{4}-\tau^{2}\right)\frac{1}{\sqrt{1+\tau^{2}u^{2}}}\right).
  \end{equation}
 It can be verified that the expression 
 \eqref{secondsummandintegrand} is equal to 
 $\frac{d}{ds}\left[-\frac{\dot{u}\sqrt{1+\tau^{2}u^{2}}}{(1+\frac{k}{4}u^{2})^{2}}\right]$. Therefore, for a rotationally invariant sphere 
 the second summand in \eqref{Erotinv} is equal to $4 \pi$.
 
 The first summand in \eqref{Erotinv} is nonnegative. It vanishes iff
 the following holds: 
 \begin{equation}
 \label{firstsummandintegrand}   
    \dot{\sigma} - \left(\frac{1}{u} + k \frac{u}{4} \right)\sin\sigma = 0.
 \end{equation} 
 
 It follows from \eqref{ode_sl2} that \eqref{firstsummandintegrand} 
 holds iff a rotationally invariant sphere is CMC.  
 Theorem \ref{thm:main2} is proved.

 \smallskip
 \smallskip
 
 \noindent  {\sc Acknowledgments}.
  The authors gratefully thank Iskander Taimanov for his original idea to study analogs of the Willmore functional in the Thurston geometries.



\begin{thebibliography}{99}

\bibitem{Weiner}
J.~L.~Weiner:
{\it On a problem of {C}hen, {W}illmore, et al},
Indiana Univ. Math. J.,
{\bf 27},
N 1 (1978),
19--35. 


\bibitem{BT05}  D.~Berdinsky, I.~Taimanov: 
    {\it Surfaces in three--dimensional Lie groups},
	Siberian Math. Journal, {\bf 46} (2005), 1005--1019. 


\bibitem{BT07} D.~Berdinsky, I.~Taimanov: 
 	{\it Surfaces of revolution in the Heisenberg group and the spectral generalization of the Willmore functional},
	Siberian Math. Journal, {\bf 48} (2007), 395--407.

\bibitem{Berd10}
D.~Berdinsky:
{\it On some generalization of the {W}illmore functional for surfaces in $\widetilde{SL}_2$},
Siberian Electronic Mathematical Reports,
{\bf 7} (2010),
140--145.	

 \bibitem{Inoguchi02}
 M.~Belkhelfa, F.~Dillen, J.~Inoguchi:
 {\it Surfaces with parallel second fundamental form in Bianchi--Cartan--Vranceanu spaces},
 PDEs, submanifolds and affine differential geometry, Banach center publ., Polish Acad. Sci.,
 {\bf 57} (2002),
 67--87. 



\bibitem{Inoguchi13}
	J.~F.~Dorfmeister, J.~Inoguchi, S~Kobayashi:
	{\it A loop group method for minimal surfaces in the three--dimensional Heisenberg group},
	arXiv:1210.7300v2 [math.DG].


\bibitem{Scott83}
P.~Scott:
{\it The geometries of 3-manifolds},
Bull. London Math. Soc.,
{\bf 15}, N 5
(1983),
401--487.


\bibitem{Da07}
B.~Daniel:
{\it Isometric immersions into 3--dimensional homogeneous manifolds},
Comment. Math. Helv.,
{\bf 82}
(2007),
87--131.



\bibitem{HH}
W.~Y.~Hsiang, W.~T.~Hsiang:
{\it On the uniqueness of isoperimetric solutions and embedded soap bubbles in noncompact symmetric spaces},
Invent. Math.,
{\bf 98}, N 1 (1989), 39--58.


\bibitem{PR}
R.~H.~L.~Pedrosa, M.~Ritor\'e:
{\it Isoperimetric domains in the Riemannian product of a circle with a simply connected space form and applications to free boundary problems},
Indiana Univ. Math. J.,
{\bf 48},
N 4 (1999),
1357--1394.

 \bibitem{Tomter93}
 P.~Tomter:
 {\it Constant mean curvature surfaces in the {H}eisenberg group},
 Proceedings of Symposia in Pure Mathematics,
 {\bf 54},
 N 1 (1993),
 485--495.


\bibitem{FMP99}
C.~Figueroa, F.~Mercuri, R.~Pedrosa:
{\it Invariant surfaces of the Heisenberg groups},
Ann. Math. Pura Appl.,
{\bf 177} (1999),
173--194.
\bibitem{CPR}
R.~Caddeo, P.~Piu, A.~Ratto:
{\it SO(2)--invariant minimal and constant mean curvature surfaces in
	$3$--dimensional homogeneous spaces},
Manuscripta Math.,
{\bf 87}
(1995),
1--12.

 \bibitem{Esp09}
 C.~Espinoza:
 {\it Rotational and parabolic surfaces in $\widetilde{PSL}_2(\mathbb{R},\tau)$ and applications},
 arXiv: 0911.2213v1 [math.DG]
 
 
 \bibitem{Torralbo10}
 F.~Torralbo:
 {\it Rotationally invariant constant mean curvature surfaces in homogeneous 3-manifolds},
 Diff. Geo. Appl.,
 {\bf 28},
 N 5 (2010),
 593--607.
 
 \bibitem{AR04}
 U.~Abresch, H.~Rosenberg:
 {\it A {H}opf differential for constant mean curvature surfaces in $\mathbb{S}^2 \times \mathbb{R}$ and $\mathbb{H}^2 \times \mathbb{R}$},
 Acta Math.,
 {\bf 193},
 N 2 (2004),
 141--174.
 
 
 \bibitem{AR05}
 U.~Abresch, H.~Rosenberg:
 {\it Generalized {H}opf differentials},
 Mat. Contemp.,
 {\bf 28},
 N 1 (2005),
 1--28.
 
  
 \bibitem{FM10}
 I.~Fern\'{a}ndez, P.~Mira:
 {\it Constant mean curvature surfaces in
 	3--dimensional {T}hurston geometries},
    arXiv:1004.4752v1 [math.DG].
 
 \bibitem{Huisken99}
 G.~Huisken, A.~Polden:
 {\it Geometric evolution equations for hypersurfaces},
 In Calculus of variations and geometric evolution problems: lectures
 given at the 2nd session of the 
 Centro Internazionale Estivo (C.I.M.E.)
 held in Cetraro, Italy, June 15-22, 1996, 45--84. Springer, 1999.
 












\end{thebibliography}

\vspace{5mm}

\begin{flushright} 

\begin{minipage}{0.45\linewidth}

 Dmitry Berdinsky
 
 Department of Computer Science 
 
 The University of Auckland
 
 Private Bag 92019 
 
 Auckland, 1142
 
 New Zealand
 
 e-mail: berdinsky@gmail.com

 \vspace{5mm}
 
 Yuri Vyatkin

 Department of Mathematics
 
 The University of Auckland
 
 Private Bag 92019 
 
 Auckland, 1142
 
 New Zealand

 e-mail: yuri.vyatkin@gmail.com

\end{minipage}

\end{flushright} 

\clearpage

\end{document}